\pdfoutput=1
\RequirePackage{ifpdf}
\ifpdf
\documentclass[10pt,pdftex]{amsart}
\else
\documentclass[10pt,dvips]{amsart}
\fi

\ifpdf
\fi
\usepackage[bookmarks, 
colorlinks=false, backref=false,plainpages=false]{hyperref}
\usepackage[english]{babel}

\usepackage[letterpaper]{geometry}

\usepackage{graphicx}
\usepackage{amsmath,amssymb,amsthm}

\usepackage{booktabs}

\newtheorem{thm}{Theorem}[section]
\newtheorem{cor}[thm]{Corollary}

\usepackage{amsaddr}

\def\XXint#1#2#3{{\setbox0=\hbox{$#1{#2#3}{\int}$}
\vcenter{\hbox{$#2#3$}}\kern-.5\wd0}}

\def\f12{\frac12}          

\newcommand{\bdm}{\begin{displaymath}}
\newcommand{\edm}{\end{displaymath}}
\newcommand{\beq}{\begin{equation}}
\newcommand{\eeq}{\end{equation}}
\newcommand{\beqa}{\begin{eqnarray}}
\newcommand{\eeqa}{\end{eqnarray}}
\newcommand{\beqas}{\begin{eqnarray*}}
\newcommand{\eeqas}{\end{eqnarray*}}

\numberwithin{equation}{section}

\title[Why the Kellogg Mesh Is Radial]{Why the Kellogg Mesh Is Radial: A Mathematical Explanation of a Classical Computational Benchmark}
\author[S. Zhang]{Shun Zhang}
\address{Department of Mathematics, City University of Hong Kong, Kowloon Tong, Hong Kong, China}
\email{shun.zhang@cityu.edu.hk}
\thanks{This work was supported in part by
Research Grants Council of the Hong Kong SAR, China, under the GRF Grant Project No. CityU 11316222, CityU 11305025}
\date{\today}

\newcommand{\dd}{\,\mathrm{d}}
\newcommand{\Tcal}{\mathcal{T}}

\newcommand{\Pcal}{\mathcal{P}}
\newcommand{\bsig}{\boldsymbol{\sigma}}
\newcommand{\Ncal}{\mathcal{N}}
\newcommand{\gradb}{\nabla_{\Pcal}}

\newcommand{\Jcal}{\mathcal{J}}

\begin{document}

\begin{abstract}
Adaptive finite element computations on Kellogg's checkerboard interface problem produce a mesh that the field reads as a check: strong refinement toward the point where the four subdomains meet, and no angular structure.  This paper explains why that reading is correct and how far it extends.  We prove that the singular solution $u(r,\theta)=r^\gamma\mu(\theta)$ satisfies the exact identities
\[
  \kappa|\nabla u|^2=\Lambda\, r^{2\gamma-2},
  \qquad
  \kappa|\gradb^2u|_F^2 =2(1-\gamma)^2\Lambda\, r^{2\gamma-4},
\]
where $|\cdot|_F$ is the Frobenius norm, $\gradb^2$ is the Hessian taken separately in each quadrant, and $\Lambda=\gamma^2\cos^2(\pi\gamma/4)$.  The point is what has disappeared: the right-hand sides depend on $r$ alone, although $\kappa$ and $u$ each depend on the angle as well.  Both identities extend across the interfaces, and the elements at the crossing point balance in the same way.  Combined with the principle of equal discretization-error distribution, they show that the target element density is radial, so a correct mesh should display nothing but refinement toward the centre.  A Kellogg mesh that is not radial is therefore visible evidence that the computation is not following the coefficient-weighted local difficulty.  The reading is specific to this benchmark: on a second interface problem, whose weighted difficulty is not radial, the same estimator correctly produces a strongly material-biased mesh, with a computed element-count ratio of $3.934$ against the predicted $4$.  A byproduct removes a practical nuisance.  The benchmark constants, normally copied between papers as sixteen-digit numbers obtained from a nonlinear solve, are values of one elementary formula, so the problem data can be generated from $\gamma$ alone at any precision.
\end{abstract}

\maketitle

\section{A familiar mesh whose explanation is missing}
\label{sec:intro}

Adaptive computations on the Kellogg checkerboard problem produce a picture, and the field has long read that picture as a check.  A convincing computation refines deeply toward the interface crossing while remaining concentric, showing no angular structure and no excess refinement along the coefficient interfaces.  A persistent preference for two opposite quadrants, or refinement driven mainly along the coefficient interfaces, has instead been treated as evidence that the method is responding to the wrong feature of the problem.  Chen and Dai's comparison of adaptive estimators made this visual distinction particularly clear \cite{ChenDai}.  Figure~\ref{fig:opening-meshes} shows the contrast in the form it usually takes.

Two questions follow, and this paper answers both.  Why is the concentric picture the correct one?  And how far does the reading extend: is a radial mesh what any coefficient-robust method produces, or is it specific to this benchmark?  We prove that the correct target element density for the Kellogg problem is exactly radial, and we show that the radial appearance is specific to it.  Figure~\ref{fig:quadratic-interface-mesh} reports the same estimator and the same adaptive loop applied to a second interface problem, in which the coefficient jumps across a single straight interface: there the computed mesh is strongly material biased, with an element-count ratio of $3.934$ between the two halves against the value $4$ predicted by the analysis below.  So the reading of Figure~\ref{fig:opening-meshes} is correct, and it is a reading of this benchmark rather than of robustness in general.  The purpose of the paper is not to discover the mesh pattern but to explain why it is the right one, and to say what it does and does not certify.

\begin{figure}[t]
\centering
\begin{minipage}[t]{0.48\textwidth}
  \centering
  \includegraphics[width=\textwidth]{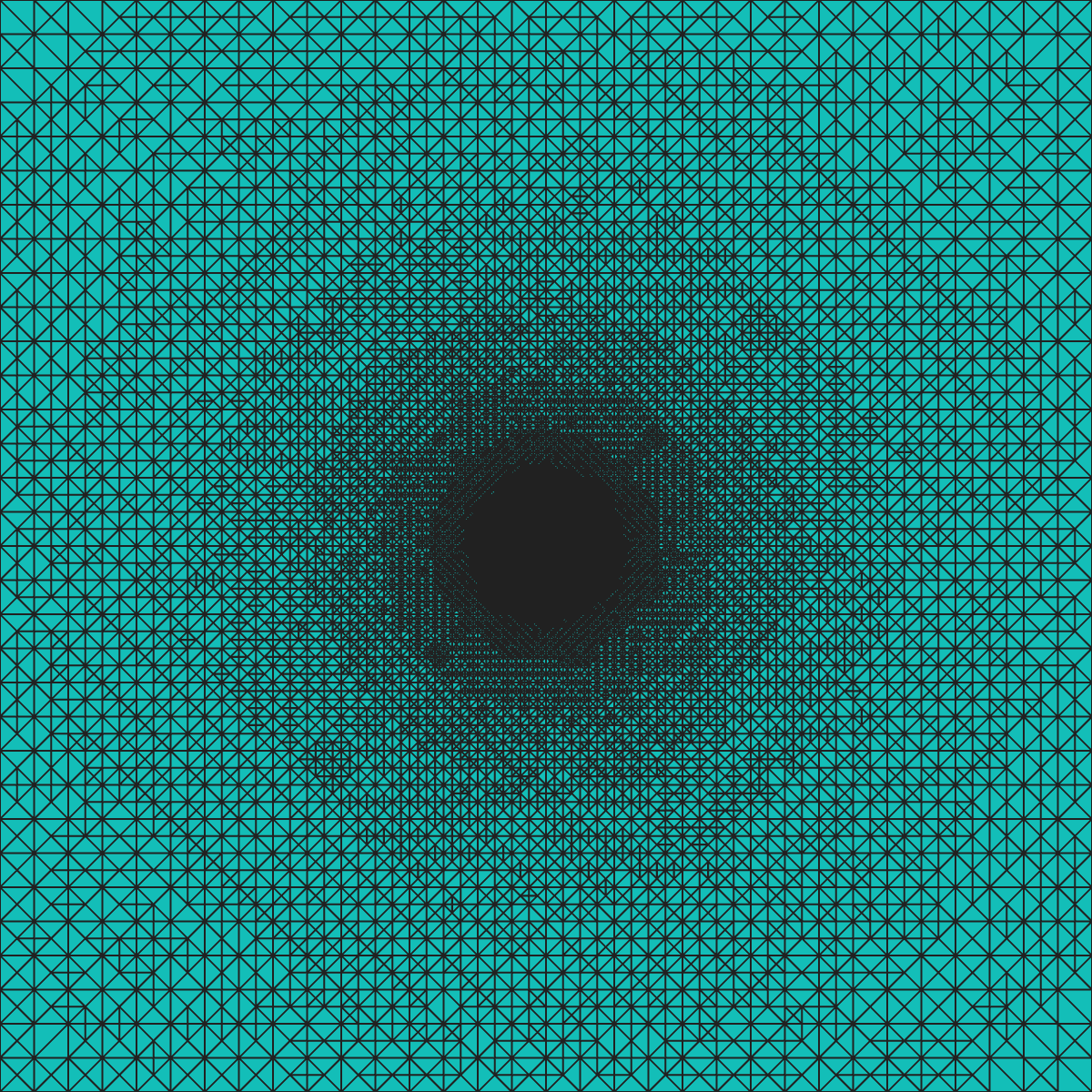}\\[-1mm]
  {\small (a) $H(\mathrm{div})$ recovery-based estimator}
\end{minipage}
\hfill
\begin{minipage}[t]{0.48\textwidth}
  \centering
  \includegraphics[width=\textwidth]{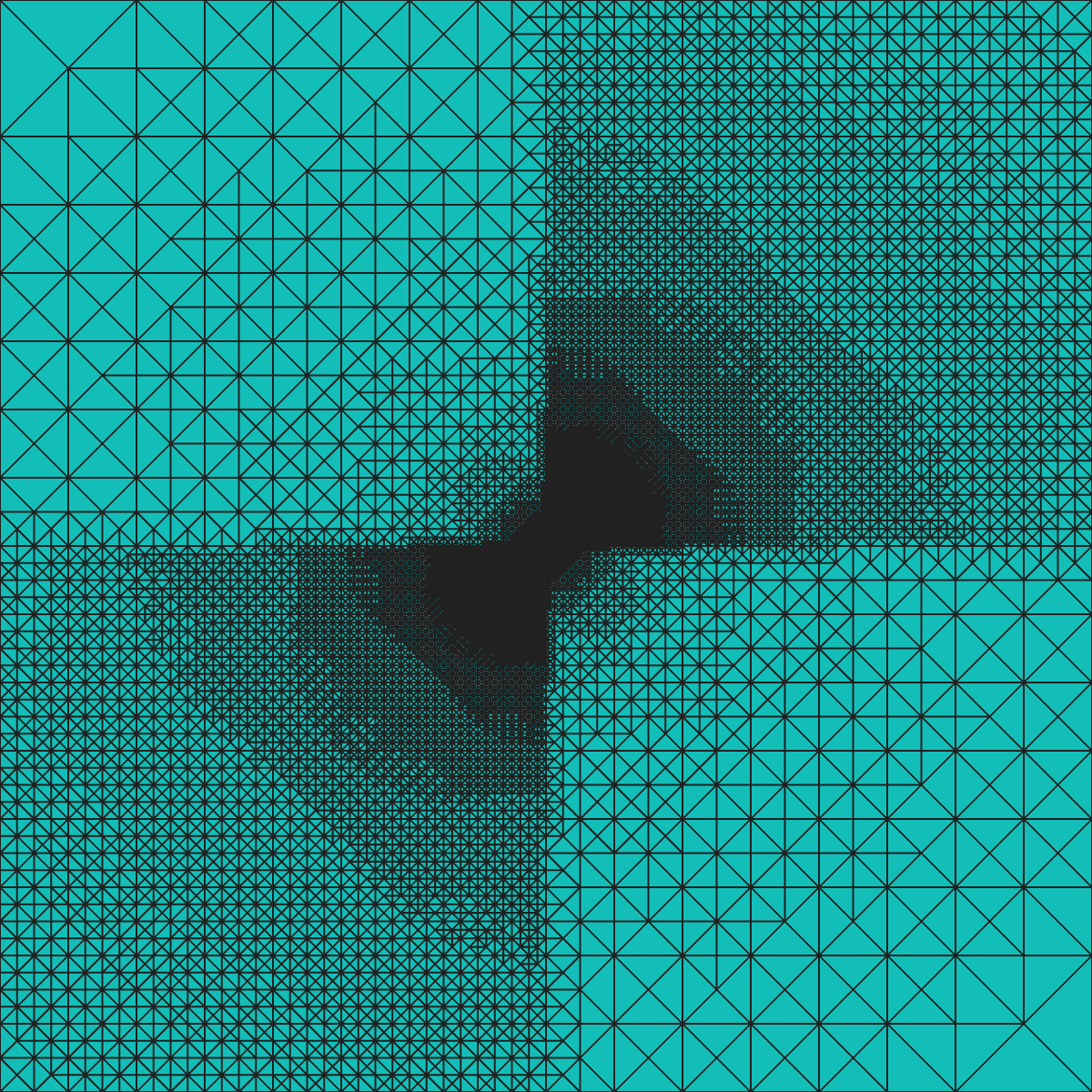}\\[-1mm]
  {\small (b) residual estimator without the coefficient weight}
\end{minipage}
\caption{
Adaptive meshes for the Kellogg problem with $\gamma=0.1$ and $R\approx161.4476$, at comparable element counts.  In (a) the mesh is driven by the $H(\mathrm{div})$ recovery-based estimator of \cite{CZ09}, which respects the coefficient weight.  In (b) it is driven by a residual estimator from which the reciprocal edge-diffusion scaling has been dropped, so that the squared edge contribution is multiplied by the corresponding edge diffusion scale.
Element counts by quadrant, in the order $Q_1,Q_2,Q_3,Q_4$, are $26\,306$, $25\,628$, $26\,347$, $25\,639$ for (a), with $103\,920$ elements in all, and $47\,308$, $5\,664$, $47\,306$, $5\,664$ for (b), with $105\,942$ elements.  Writing the ratio of the two high-coefficient quadrants to the two low-coefficient ones, (a) gives $1.03$ and (b) gives $8.35$.  Note also that in (a) there is no visible excess refinement along the coefficient interfaces: away from the crossing the refinement is concentric.
}
\label{fig:opening-meshes}
\end{figure}

\begin{figure}[t]
\centering
\includegraphics[width=0.50\textwidth]{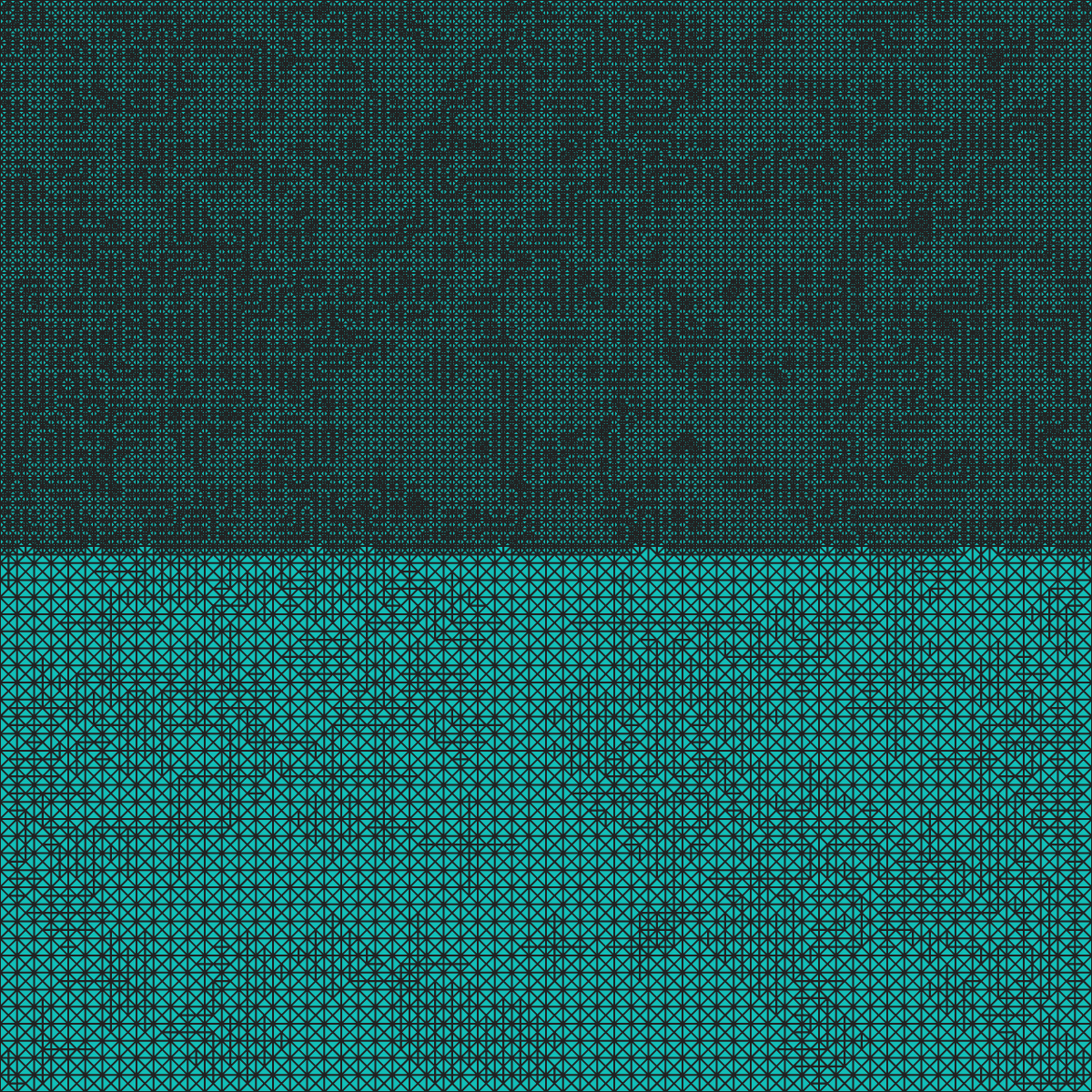}
\caption{
Adaptive mesh for a second interface problem, in which the coefficient jumps by a factor $q=16$ across the single interface $x_2=0$ and the weighted difficulty is not radial; the problem is stated in section~\ref{sec:symmetry}.  The mesh was computed with the estimator and adaptive loop of Figure~\ref{fig:opening-meshes}(a).  The lower half-square carries the smaller coefficient and the upper the larger.  At $52\,790$ elements they contain $10\,700$ and $42\,090$ elements, an observed ratio of $3.934$ against the predicted $\sqrt q=4$.
}
\label{fig:quadratic-interface-mesh}
\end{figure}

For completeness, we recall the benchmark.  Let $\Omega=(-1,1)^2$ and let the diffusion coefficient be $\kappa(x)=R$ on the first and third quadrants and $\kappa(x)=1$ on the second and fourth.  Consider
\begin{equation}\label{eq:cd-problem}
  -\nabla\cdot(\kappa\nabla u)=0 \quad\hbox{in }\Omega,
  \qquad
  u=g:=u_{\rm exact}\big|_{\partial\Omega} \quad\hbox{on }\partial\Omega ,
\end{equation}
whose exact solution is $u(r,\theta)=r^\gamma\mu(\theta)$ with
\begin{equation}\label{eq:cd-mu}
  \mu(\theta)=
  \begin{cases}
   \cos\bigl((\tfrac\pi2-\sigma)\gamma\bigr)\,
   \cos\bigl((\theta-\tfrac\pi2+\rho)\gamma\bigr),
     & 0\le\theta\le\tfrac\pi2,\\[1mm]
   \cos(\rho\gamma)\,\cos\bigl((\theta-\pi+\sigma)\gamma\bigr),
     & \tfrac\pi2\le\theta\le\pi,\\[1mm]
   \cos(\sigma\gamma)\,\cos\bigl((\theta-\pi-\rho)\gamma\bigr),
     & \pi\le\theta\le\tfrac{3\pi}2,\\[1mm]
   \cos\bigl((\tfrac\pi2-\rho)\gamma\bigr)\,
   \cos\bigl((\theta-\tfrac{3\pi}2-\sigma)\gamma\bigr),
     & \tfrac{3\pi}2\le\theta\le2\pi,
  \end{cases}
\end{equation}
where the four numbers $\gamma,\rho,\sigma,R$ must satisfy the nonlinear system
\begin{equation}\label{eq:cd-system}
  \left\{
  \begin{aligned}
   R&=-\tan\bigl((\tfrac\pi2-\sigma)\gamma\bigr)\cot(\rho\gamma),\\
   1/R&=-\tan(\rho\gamma)\cot(\sigma\gamma),\\
   R&=-\tan(\sigma\gamma)\cot\bigl((\tfrac\pi2-\rho)\gamma\bigr),
  \end{aligned}\right.
\end{equation}
together with the admissibility restrictions $0<\gamma<2$, $\max(0,\pi\gamma-\pi)<2\gamma\rho<\min(\pi\gamma,\pi)$, and $\max(0,\pi-\pi\gamma)<-2\gamma\sigma<\min(\pi,2\pi-\pi\gamma)$.  These are Kellogg's matching relations for perpendicular interfaces \cite{Kellogg}, in the form used by Chen and Dai \cite{ChenDai}.

For the standard choice $\gamma=0.1$ the system is traditionally solved numerically, yielding
\begin{equation}\label{eq:cd-numbers}
  R\approx161.4476387975881,
  \qquad \rho=\pi/4,
  \qquad \sigma\approx-14.92256510455152 .
\end{equation}

Two observations about this formulation will be useful below.  First, the constants \eqref{eq:cd-numbers} look arbitrary and are usually copied from one paper to the next.  Corollary~\ref{cor:closed} below shows that they are values of one elementary formula,
\begin{equation}\label{eq:closed-intro}
  R=\cot^2\!\left(\frac{\pi\gamma}{4}\right),
  \qquad
  \rho=\frac{\pi}{4},
  \qquad
  \sigma=\frac{\pi}{4}-\frac{\pi}{2\gamma},
\end{equation}
so no nonlinear solve is needed to set up the benchmark and its data can be generated from $\gamma$ alone at any precision. Second, nothing in \eqref{eq:cd-mu}--\eqref{eq:cd-numbers} suggests any symmetry: the coefficient is asymmetric under a quarter turn, and so is $\mu$.  The point of this paper is that the coefficient-weighted local difficulty built from $u$ is nevertheless radially symmetric.

Kellogg's original work \cite{Kellogg} concerned regularity for elliptic equations with intersecting interfaces, and modes of this homogeneous form are what it produced.  At $\gamma=0.1$ the solution lies in $H^{1+\delta}(\Omega)$ only for $\delta<\gamma$, so it is severely singular; for larger $\gamma$ the global exponent is limited by the gradient jumps across the interfaces as well, and the analysis below uses elementwise regularity rather than any global exponent.  For the two-phase checkerboard the contrast $R$ and the exponent $\gamma$ are tied by a single formula, so an arbitrarily large jump and an arbitrarily low regularity are the same statement.  A distinguishing feature of the benchmark is that the singularity lies at the interior crossing of the four subdomains.  The energy therefore accumulates precisely where the coefficient is least well behaved.  This combination of low regularity, large contrast, and intersecting interfaces cannot be separated and made mild one factor at a time, which is what makes the problem a standard test for robustness.

What robustness asks of a method is easy to state.  Since the energy norm $\|\kappa^{1/2}\nabla\cdot\|_0$ is fixed by the operator itself, a discretization is \emph{robust} when it respects that norm and its coefficient weighting, with constants independent of the coefficient contrast.  A substantial literature has developed coefficient-robust analysis for several discretizations and estimator families.  The requirement that the constants be independent of the jump in $\kappa$ goes back to Bernardi and Verf\"urth \cite{BV}, whose residual analysis for discontinuous coefficients set the pattern, and was developed further by Petzoldt \cite{Petzoldt}.  Recovery-based estimators were built for conforming, mixed and nonconforming elements, and then sharpened \cite{CZ09,CZ10,CHZzz17}; residual-based estimators were carried to the Crouzeix--Raviart and discontinuous Galerkin approximations, together with the matching a priori theory \cite{CYZ11,CHZmc17,CHZ17}; and equilibrated-residual estimators were made robust for conforming elements \cite{CZ12,ZhangThesis}.  In parallel, Ern, Vohral\'ik and their coauthors pursued guaranteed and fully robust bounds by reconstructing an $H(\mathrm{div})$ flux, with stability constants independent of the diffusion heterogeneities \cite{ENV07,EV09,ESV10,Vohralik2011,DiPietroErn}. Robust flux reconstruction for discontinuous coefficients has been taken up more recently by Capatina, Gouasmi and He \cite{CGH24}.

The example itself entered the adaptive literature through Morin, Nochetto and Siebert \cite{MNS}, and the computations of Chen and Dai \cite{ChenDai} were the ones that displayed the correct mesh; it is largely through the latter that the benchmark passed into the work on robust estimators.  It has since spread further, into broader benchmark collections and newer approximation technologies \cite{Mitchell,IGAActa,LiuCai2022}.

The same robustness requirement also has consequences for spatial approximation. Respecting the weighted norm already forced a refinement of the a priori theory: a bound written with a single global regularity exponent does not distinguish where the approximation difficulty is located, and for this example that difficulty is extremely uneven: the solution has very low regularity on the elements touching the origin and is analytic away from it. Replacing the global exponent by elementwise exponents yields the locally optimal estimate, established for the Crouzeix--Raviart and discontinuous Galerkin approximations in \cite{CHZ17} and for the Raviart--Thomas and Brezzi--Douglas--Marini mixed approximations in \cite{Z20}, and that estimate is the basis on which adaptive methods pursue equal discretization-error distribution \cite{NV2012}.  Carried one step further, the same requirement predicts the spatial structure of the mesh.  This last step appears not to have been taken, and it is the subject of this paper.

Previous studies of the Kellogg problem have focused mainly on its convergence history, as a demonstration that the optimal energy decay $(\#\Tcal)^{-1/2}$ is recovered in the very singular case $\gamma=0.1$; this message runs from Morin--Nochetto--Siebert \cite{MNS}, through the survey of Nochetto--Siebert--Veeser \cite{NSV2009}, to the 2024 Acta Numerica review of Bonito--Canuto--Nochetto--Veeser \cite{BCNV2024}.  Recovering that rate on this benchmark is a demanding achievement, and it is natural that nothing further was asked.  But a global decay statement does not locate the degrees of freedom:
\begin{equation}\label{eq:rate-not-mesh}
  \boxed{\text{rate optimality does not determine spatial allocation}.}
\end{equation}
This distinction explains why a further question stayed open: why should the spatially correct Kellogg mesh be centered, with no refinement signature attached to the coefficient interfaces?

The usual regularity statement
\[
u\sim r^\gamma,\qquad |\nabla u|\sim r^{\gamma-1}
\]
explains why refinement must accumulate at the origin, but it does not explain how that refinement should be distributed in angle.  In fact, the coefficient and the exact solution are both strongly asymmetric under a quarter turn.  For the classical choice $\gamma=0.1$, the contrast is about $161.45$, and neighboring sector amplitudes differ by a factor of order $\sqrt R$.  There is therefore no obvious reason, at the level of $u$ or $|\nabla u|$, for a balanced mesh.

The missing fact is that an exact coefficient-weighted quantity is radial.  We prove that Kellogg's homogeneous mode satisfies
\begin{equation}\label{eq:intro-main}
 \kappa|\nabla u|^2=\Lambda\, r^{2\gamma-2},
 \qquad
 \kappa|\gradb^2u|_F^2=2(1-\gamma)^2\Lambda\, r^{2\gamma-4},
\end{equation}
where $|\cdot|_F$ denotes the Frobenius norm, $|M|_F^2=\sum_{i,j}M_{ij}^2$, and $\gradb^2$ is the Hessian broken over the partition $\Pcal=\{Q_1,Q_2,Q_3,Q_4\}$: since $\nabla u$ jumps across the interfaces, second derivatives of $u$ exist only piecewise, and $\gradb^2u$ is the matrix of second derivatives formed inside each $Q_i$ separately.  The first identity is an exact radial energy law: although neither $\kappa$ nor $u$ is invariant under a quarter turn, their combination is.  Although $\kappa$, $\nabla u$ and the broken Hessian each have distinct one-sided traces on the interfaces, the two weighted combinations have identical ones, so \eqref{eq:intro-main} holds at every point except the origin.  In particular the weighted difficulty is continuous across the interfaces.

The essential point is what \eqref{eq:intro-main} does \emph{not} contain.  Both right-hand sides are functions of $r$ alone: the angle has disappeared, and so has the coefficient, although $\kappa$ and $u$ each depend on both.  This is what makes the benchmark readable by eye.  The ideal local density is radial, nothing in it distinguishes one direction from another, and the interfaces carry no refinement signature of their own.  A correct mesh should therefore show nothing but refinement toward the center.  The second identity gives the radial piecewise-$H^2$ approximation measure on every element that does not touch the origin.  On the few that do, the $H^2$ seminorm is not integrable and the corresponding local scale is supplied by fractional regularity; the same coefficient-weighted amplitude balance gives that contribution the same form in all four quadrants.

We then use the usual principle of \emph{equal discretization-error distribution}: an ideally adapted mesh makes the local discretization errors comparable \cite{NV2012,CHZ17}.  Once the local error scale is shown to depend on the radius alone, equidistribution forces the mesh to have the same property.  The argument below therefore proceeds in three steps: the variational problem identifies the correct coefficient weight and the elementwise form of the robustness requirement; the Kellogg matching makes that weighted quantity radial; and equidistribution converts radiality into the familiar adaptive mesh.

\section{What robustness requires}
\label{sec:energy}

Consider
\begin{equation}\label{eq:diffusion}
  -\nabla\cdot(\kappa\nabla p)=f \quad\hbox{in }\Omega,
  \qquad p=0 \quad\hbox{on }\partial\Omega,
\end{equation}
where $\kappa$ is positive and piecewise constant.  Here $p$ denotes the solution of this generic model problem; the letter $u$ remains reserved for the Kellogg solution of \eqref{eq:cd-problem}.  Homogeneous data are assumed only so that the admissible space is $H_0^1(\Omega)$; for data such as \eqref{eq:cd-problem} the same computation runs on the corresponding affine space.  The weak problem is
\begin{equation}\label{eq:weak}
  (\kappa\nabla p,\nabla v)=(f,v)
  \qquad\forall v\in H_0^1(\Omega),
\end{equation}
and $p$ is equivalently the minimizer over $H_0^1(\Omega)$ of the energy
\begin{equation}\label{eq:energy-functional}
  \Jcal(v):=\tfrac12(\kappa\nabla v,\nabla v)-(f,v).
\end{equation}

It is the minimization form that should be kept in view, because it fixes the norm before any method is chosen.  Completing the square in \eqref{eq:energy-functional} and using \eqref{eq:weak},
\begin{equation}\label{eq:complete-square}
  \Jcal(v)=\tfrac12\bigl\|\kappa^{1/2}\nabla(v-p)\bigr\|_0^2
          -\tfrac12\bigl\|\kappa^{1/2}\nabla p\bigr\|_0^2
  \qquad\forall v\in H_0^1(\Omega),
\end{equation}
and the second term does not depend on $v$.  Hence for \emph{any} set $S\subset H_0^1(\Omega)$,
\begin{equation}\label{eq:ritz-best}
  \operatorname*{arg\,min}_{v\in S}\Jcal(v)
  =\operatorname*{arg\,min}_{v\in S}\bigl\|\kappa^{1/2}\nabla(p-v)\bigr\|_0 .
\end{equation}
Minimizing the energy \emph{is} minimizing the coefficient-weighted error.  This uses no linearity of $S$, no orthogonality relation, and no finite element structure.  Nonconforming, discontinuous Galerkin and mixed discretizations do not sit inside $H_0^1(\Omega)$ and measure their errors in their own norms, on the broken gradient or on the flux, but those norms are built to reproduce the same coefficient weighting.  What \eqref{eq:ritz-best} identifies is the continuous scale they all have to match.  When $S=V_h$ is a subspace, \eqref{eq:ritz-best} specializes to the familiar statement that the Galerkin solution $p_h\in V_h$ is the best approximation,
\begin{equation}\label{eq:best}
  \|\kappa^{1/2}\nabla(p-p_h)\|_0
  =\inf_{v_h\in V_h}\|\kappa^{1/2}\nabla(p-v_h)\|_0 .
\end{equation}

The weight $\kappa^{1/2}$ used throughout is therefore dictated by the variational structure, rather than introduced for the subsequent analysis.  A discretization deserves to be called robust when it respects this energy norm and its coefficient weighting, with constants independent of the coefficient contrast.  The mesh question of this paper puts that same requirement to the triangulation.

For the mesh question, robustness alone is not enough; one also needs its locally optimal counterpart.  From here on we return to the Kellogg pair $(\kappa,u)$ of \eqref{eq:cd-problem}.  Let $\Tcal$ be a shape-regular, interface-fitted triangulation of $\Omega$ without hanging nodes, with $h_K=\operatorname{diam}K$ for $K\in\Tcal$; interface fitted means that the coordinate axes are unions of edges, so each $K$ lies in one closed quadrant and $\kappa|_K$ is constant.  A bound in which the regularity of $u$ enters through one global exponent may be robust, but it is blunt, and for the present mesh question it is useless: it does not distinguish where the approximation difficulty is located.  For the Kellogg solution that difficulty is extremely uneven, since the solution has very low regularity on the elements touching the origin and is analytic elsewhere.  The locally optimal form replaces the global exponent by elementwise exponents.  If $u|_K\in H^{1+s_K}(K)$ with $0<s_K\le1$, which is the only range a linear element uses, the resulting local approximation scale is
\begin{equation}\label{eq:local-scale}
  \varepsilon_K
  :=h_K^{s_K}\kappa_K^{1/2}|\nabla u|_{s_K,K},
\end{equation}
where $|\cdot|_{s,K}$ is the Sobolev seminorm of order $s$ applied to the vector field $\nabla u$; for $s=1$ this is $\bigl(\int_K|\nabla^2u|_F^2\bigr)^{1/2}$, equivalent to the usual scalar $H^2$ seminorm.  Writing the scale through $\nabla u$ rather than through $u$ is what makes the fractional case available.  Local scales of this form go back to the local polynomial approximation theory of Dupont and Scott \cite{DupontScott}.  Locally optimal, coefficient-robust approximation estimates built on them are available for several finite element families in two dimensions: for conforming elements when the local regularity permits nodal interpolation, for the Crouzeix--Raviart and discontinuous Galerkin approximations \cite{CHZ17}, and for the mixed approximations \cite{Z20}.  The discrete error variable and the norm in which it is measured differ from family to family, but in each case the local quantity controlling the estimate is $\varepsilon_K$, with a constant \emph{independent of the jump in $\kappa$} and depending only on the usual approximation parameters.

The local scale \eqref{eq:local-scale} is also the quantity on which adaptivity acts: it is the local error scale that an ideally adapted mesh should equidistribute \cite{CHZ17,NV2012}.  Everything below uses \eqref{eq:local-scale} in that role only: no discrete solution, no particular finite element family, and no marking strategy enters the mesh derivation.

On an interface-fitted element not touching the origin, $s_K=1$ and
\begin{equation}\label{eq:p1-local}
  \varepsilon_K^2
  =h_K^2\kappa_K|\nabla u|_{1,K}^2
  =h_K^2\int_K\kappa\,|\nabla^2u|_F^2\,\dd x,
  \qquad
  |M|_F^2:=\sum_{i,j}M_{ij}^2 .
\end{equation}
Here $|\cdot|_F$ denotes the Frobenius (entrywise Euclidean) norm.  The seminorm $|\nabla u|_{1,K}$ is by definition the sum over all second partial derivatives, $|\nabla u|_{1,K}^2=\sum_{i,j}\int_K|\partial_{ij}u|^2$, which is exactly $\int_K|\nabla^2u|_F^2$; the mixed derivative is counted with the multiplicity it has in that double sum.  The constant $\kappa_K$ was taken inside the integral because $\kappa|_K\equiv\kappa_K$ on an interface-fitted element.  For the same reason the Hessian and the broken Hessian coincide on such an element, so we write $\nabla^2u$ in elementwise statements and $\gradb^2u$ in global ones.  On the finitely many elements whose closure contains the origin, \eqref{eq:p1-local} is unavailable and one must use $s_K<\gamma$ in \eqref{eq:local-scale} instead.  The mesh argument therefore needs two local measures, a weighted broken $H^2$ one away from the origin and a weighted fractional one at it, and the checkerboard symmetry proved next balances both.

\section{Kellogg's checkerboard and its hidden symmetry}
\label{sec:symmetry}

We keep the data of \eqref{eq:cd-problem}: on $\Omega=(-1,1)^2$ with quadrants $Q_1,\ldots,Q_4$,
\begin{equation}\label{eq:kappa}
 \kappa(x)=
 \begin{cases}
 R,&x\in Q_1\cup Q_3,\\
 1,&x\in Q_2\cup Q_4,
 \end{cases}
 \qquad R>1.
\end{equation}
The interface conditions follow from membership in the two natural spaces: the potential satisfies $u\in H^1(\Omega)$ and the flux
\begin{equation}\label{eq:flux}
  \bsig:=-\kappa\nabla u
\end{equation}
satisfies $\bsig\in H(\mathrm{div};\Omega)$, since $\nabla\cdot\bsig=0$ by \eqref{eq:cd-problem}.  Traces of $H^1$ functions match across an interface, and normal traces of $H(\mathrm{div})$ fields match across an interface, so
\begin{equation}\label{eq:transmission}
 [u]=0,\qquad [\bsig\cdot n]=-[\kappa\partial_nu]=0 .
\end{equation}
The two continuities are of different objects, the potential and the normal flux.
For a homogeneous singular mode
\begin{equation}\label{eq:mode}
 u(r,\theta)=r^\gamma\mu(\theta),\qquad 0<\gamma<1,
\end{equation}
the coefficient is constant inside each open quadrant, so $\nabla\cdot(\kappa\nabla u)=\kappa\Delta u$ there and $u$ is harmonic.  In polar coordinates $\Delta u=u_{rr}+r^{-1}u_r+r^{-2}u_{\theta\theta}$, and substituting \eqref{eq:mode},
\begin{equation}\label{eq:laplacian}
  \Delta u
  =r^{\gamma-2}\Bigl[\underbrace{\gamma(\gamma-1)+\gamma}_{=\,\gamma^2}\,\mu+\mu''\Bigr]
  =r^{\gamma-2}\bigl(\mu''+\gamma^2\mu\bigr),
\end{equation}
so $\mu''+\gamma^2\mu=0$ in each quadrant.  Its general solution there is $a\cos(\gamma\theta)+b\sin(\gamma\theta)$, and writing $(a,b)=A_i(\cos\gamma\theta_i,\sin\gamma\theta_i)$ puts it in amplitude and phase form,
\begin{equation}\label{eq:sector}
 \mu_i(\theta)=A_i\cos\bigl(\gamma(\theta-\theta_i)\bigr).
\end{equation}
Thus $\gamma$ appears as the angular frequency.  The amplitude $A_i$ and the phase $\theta_i$ are not free parameters: they are determined, up to a single overall multiplicative constant, by the interface conditions \eqref{eq:angular-transmission} imposed below, and computing them is exactly what the transfer-matrix argument does.

The interfaces are rays, so the interface normal is $\pm e_\theta$ and $\partial_nu=\pm r^{-1}\partial_\theta u=\pm r^{\gamma-1}\mu'$; once a common orientation is fixed the sign is the same on both sides and cancels in the flux matching.  Conditions \eqref{eq:transmission} therefore read $r^{\gamma}[\mu]=0$ and $r^{\gamma-1}[\kappa\mu']=0$.  The powers of $r$ are the same on the two sides of an interface and cancel identically in $r$, which is why a single power $r^\gamma$ can meet the interface conditions at all radii simultaneously.  What survives is purely angular:
\begin{equation}\label{eq:angular-transmission}
  [\mu]=0,\qquad [\,\tau\,]=0,\qquad \tau:=\kappa\mu' .
\end{equation}
These are the angular traces of the two continuities in \eqref{eq:transmission}: $\mu$ carries the trace of $u$ and $\tau$ carries the normal trace of $\bsig$, up to the common power of $r$.  Note that $\mu'$ itself is \emph{not} continuous across an interface.  Every asymmetry below originates there.  Together with \eqref{eq:laplacian}, these are the matching conditions solved by \eqref{eq:cd-mu}--\eqref{eq:cd-system}.

The first cancellation is purely local.  In the orthonormal polar frame $\nabla u=u_r\,e_r+r^{-1}u_\theta\,e_\theta$, so \eqref{eq:mode} gives $\nabla u=\gamma r^{\gamma-1}\mu\,e_r+r^{\gamma-1}\mu'\,e_\theta$ and, in any one sector,
\begin{equation}\label{eq:sector-gradient}
 |\nabla u|^2
 =r^{2\gamma-2}\bigl(\gamma^2\mu^2+(\mu')^2\bigr)
 =\gamma^2A_i^2r^{2\gamma-2}.
\end{equation}
The second equality is the Pythagorean identity applied to \eqref{eq:sector}, but the structural reason is visible without solving: by \eqref{eq:laplacian},
\begin{equation}\label{eq:conservation}
  \frac{\dd}{\dd\theta}\bigl(\gamma^2\mu^2+(\mu')^2\bigr)
  =2\mu'\bigl(\mu''+\gamma^2\mu\bigr)=0 ,
\end{equation}
so $\gamma^2\mu^2+(\mu')^2$ is the conserved energy of the angular oscillator.  For the Hessian, fix the branch of $z^\gamma$ on the sector under consideration, which is possible since the sector is simply connected and does not contain the origin in its interior.  With $z=x_1+ix_2=re^{i\theta}$ and $\lambda=A_ie^{-i\gamma\theta_i}$ one has $\operatorname{Re}(\lambda z^\gamma)=A_ir^\gamma\cos\bigl(\gamma(\theta-\theta_i)\bigr)$, which by \eqref{eq:sector} is $u$ on that sector.  So $u=\operatorname{Re}(\lambda z^\gamma)$ there, with $|\lambda|=|A_i|$.  For any $\varphi=\operatorname{Re}\Phi$ with $\Phi$ holomorphic, $\varphi_{x_1x_1}=\operatorname{Re}\Phi''$, $\varphi_{x_1x_2}=-\operatorname{Im}\Phi''$ and $\varphi_{x_2x_2}=-\operatorname{Re}\Phi''$, so $\nabla^2\varphi$ is trace free and, summing the four entries, $|\nabla^2\varphi|_F^2=2|\Phi''|^2$.  With $\Phi=\lambda z^\gamma$ and $\Phi''=\gamma(\gamma-1)\lambda z^{\gamma-2}$,
\begin{equation}\label{eq:sector-hessian}
 |\nabla^2u|_F^2
 =2\gamma^2(1-\gamma)^2A_i^2r^{2\gamma-4}.
\end{equation}
Thus all angular dependence disappears inside a quadrant.  The only issue is whether the weighted amplitudes $\kappa_iA_i^2$ agree from quadrant to quadrant.

That agreement is a consequence of the special perpendicular matching.  Set $Y=(\mu,\tau)^T$ with $\tau=\kappa\mu'$ as in \eqref{eq:angular-transmission}, so that both components of $Y$ are continuous across every interface, the first because $u\in H^1(\Omega)$ and the second because $\bsig\in H(\mathrm{div};\Omega)$.  This is why we work with $Y$ rather than with $(\mu,\mu')^T$.  Inside a sector $\kappa$ is constant, so $\mu''+\gamma^2\mu=0$ there and a solution is determined by the pair $(\mu,\mu')$ at any one point.  The two particular solutions $\cos(\gamma s)$ and $\gamma^{-1}\sin(\gamma s)$ have data $(1,0)$ and $(0,1)$ at $s=0$, so the solution carrying the data $(\mu(\theta),\mu'(\theta))$ is their combination with exactly those coefficients,
\[
\mu(\theta+s)=\cos(\gamma s)\,\mu(\theta)+\frac{\sin(\gamma s)}{\gamma}\,\mu'(\theta),
\qquad
\mu'(\theta+s)=-\gamma\sin(\gamma s)\,\mu(\theta)+\cos(\gamma s)\,\mu'(\theta),
\]
the second line being the $s$-derivative of the first.  Setting $s=\omega$ transports the data across a sector of width $\omega$, and eliminating $\mu'$ in favor of $\tau=\kappa\mu'$ gives, in a sector of coefficient $\kappa$,
\begin{equation}\label{eq:transfer}
 Y(\theta+\omega)=M_\kappa(\omega)Y(\theta),\qquad
 M_\kappa(\omega)=
 \begin{pmatrix}
 \cos(\gamma\omega)&\dfrac{\sin(\gamma\omega)}{\gamma\kappa}\\[2mm]
 -\gamma\kappa\sin(\gamma\omega)&\cos(\gamma\omega)
 \end{pmatrix}.
\end{equation}
The two off-diagonal entries depend reciprocally on $\kappa$, and therefore $\det M_\kappa(\omega)=\cos^2(\gamma\omega)+\sin^2(\gamma\omega)=1$.

The interface conditions enter here.  Both components of $Y$ are continuous across an interface by \eqref{eq:angular-transmission}, so the sector matrices compose with nothing in between: transporting through $Q_1$ and then through $Q_2$ carries $Y(0)$ to $Y(\pi)$ by the single matrix
\begin{equation}\label{eq:halfturn}
  N:=M_1(\pi/2)\,M_R(\pi/2).
\end{equation}
Had $\mu$ or $\tau$ jumped at $\theta=\pi/2$, a jump matrix would stand between the two factors and \eqref{eq:halfturn} would not be a product of sector matrices alone.  Multiplying out \eqref{eq:transfer} twice, and abbreviating $c=\cos(\pi\gamma/2)$ and $s=\sin(\pi\gamma/2)$,
\begin{equation}\label{eq:Nexplicit}
  N=\begin{pmatrix}
  c^2-Rs^2 & \dfrac{cs}{\gamma}\Bigl(1+\dfrac1R\Bigr)\\[3mm]
  -\gamma cs(1+R) & c^2-\dfrac{s^2}{R}
  \end{pmatrix},
\end{equation}
from which $\det N=1$ and $\operatorname{tr}N=2c^2-s^2(R+R^{-1})$ may be read off directly.

Since $\kappa$ is $\pi$-periodic, the same matrix carries $Y(\pi)$ to $Y(2\pi)$, so $N^2Y(0)=Y(2\pi)$.  Single-valuedness of $u$ gives $\mu(2\pi)=\mu(0)$, and continuity of $\tau$ across the closing interface $\theta=0$ gives $\tau(2\pi)=\tau(0)$.  Hence
\begin{equation}\label{eq:closing}
  N^2Y(0)=Y(0),
\end{equation}
and the two conditions used are exactly the two continuities of \eqref{eq:transmission}, one for each component.

The state is then an eigenvector of $N$ itself.  Suppose $Y(0)$ and $NY(0)$ were linearly independent.  By \eqref{eq:closing}, $N^2$ fixes $Y(0)$, and it also fixes $NY(0)$, since $N^2(NY(0))=N(N^2Y(0))=NY(0)$; so $N^2$ would fix a basis and $N^2=I$.  With $\det N=1$ this forces $N=\pm I$.  But the upper right entry of \eqref{eq:Nexplicit} is $\gamma^{-1}cs(1+R^{-1})$, which does not vanish for $0<\gamma<1$.  Hence $NY(0)=\lambda Y(0)$, and \eqref{eq:closing} gives $\lambda^2=1$.  Since $\det N=1$, the other eigenvalue is $\lambda$ as well, so $\operatorname{tr}N=2\lambda$.  By \eqref{eq:Nexplicit} and $c^2=1-s^2$,
\begin{equation}\label{eq:traceN}
  \operatorname{tr}N
  =2-\sin^2\Bigl(\frac{\pi\gamma}{2}\Bigr)\Bigl(2+R+\frac1R\Bigr)<2,
\end{equation}
strictly, since $0<\gamma<1$ and $R>1$.  Now $\lambda^2=1$ leaves only $\operatorname{tr}N=2\lambda=\pm2$, and \eqref{eq:traceN} excludes $+2$.  Therefore
\[
  \operatorname{tr}N=-2,\qquad \lambda=-1,
\]
and the mode is antiperiodic over a half turn.  Setting \eqref{eq:traceN} equal to $-2$ gives $\sin(\pi\gamma/2)=2\sqrt R/(1+R)$ and $\cos(\pi\gamma/2)=(R-1)/(R+1)$, the positive branches being the right ones since $0<\gamma<1$ places $\pi\gamma/2$ in $(0,\pi/2)$ and $R>1$, hence $\tan(\pi\gamma/4)=1/\sqrt R$, that is,
\begin{equation}\label{eq:Rgamma}
  R=\cot^2\left(\frac{\pi\gamma}{4}\right).
\end{equation}

With $\operatorname{tr}N=-2$ and $\det N=1$, both eigenvalues of $N$ equal $-1$, so $N+I$ is nilpotent; it is not the zero matrix, because of the off-diagonal entry in \eqref{eq:Nexplicit}.  Its kernel is therefore one dimensional: for a given contrast the mode is unique up to a multiplicative constant.  Reading the first row of $(N+I)Y(0)=0$ and simplifying with \eqref{eq:Rgamma} gives $\tau(0)=\gamma\sqrt R\,\mu(0)$; transporting through $Q_1$ gives $\tau(\pi/2)=-\gamma\sqrt R\,\mu(\pi/2)$, and antiperiodicity repeats the pattern at $\theta=\pi$ and $\theta=3\pi/2$.  Thus at every interface
\begin{equation}\label{eq:interface-eigenvector}
  |\tau|=\gamma\sqrt R\,|\mu|=\gamma\sqrt{\kappa_+\kappa_-}\,|\mu| .
\end{equation}

Consider an interface separating coefficients $\kappa_i$ and $\kappa_j$, and evaluate everything at a point of that interface, where $\mu$ and $\tau$ take common values by \eqref{eq:angular-transmission}.

Approach the point from the $\kappa_i$ side.  There $\tau=\kappa_i\mu'$, so $\mu'=\tau/\kappa_i$, and the conserved quantity \eqref{eq:conservation}, whose value is $\gamma^2A_i^2$ by \eqref{eq:sector}, reads
\[
  \gamma^2A_i^2=\gamma^2\mu^2+(\mu')^2=\gamma^2\mu^2+\frac{\tau^2}{\kappa_i^2}.
\]
Multiplying by $\kappa_i$ gives
\[
  \kappa_i\gamma^2A_i^2=\kappa_i\gamma^2\mu^2+\kappa_i\frac{\tau^2}{\kappa_i^2}
  =\kappa_i\gamma^2\mu^2+\frac{\tau^2}{\kappa_i},
\]
and squaring \eqref{eq:interface-eigenvector} to $\tau^2=\gamma^2\kappa_i\kappa_j\mu^2$ turns the last term into $\gamma^2\kappa_j\mu^2$, so
\begin{equation}\label{eq:balance-step}
  \kappa_i\gamma^2A_i^2=\gamma^2(\kappa_i+\kappa_j)\mu^2 .
\end{equation}
The right-hand side is symmetric in $i$ and $j$.  Approaching the same point from the $\kappa_j$ side uses $\mu'=\tau/\kappa_j$ and the same $\mu$ and $\tau$, so it gives $\kappa_j\gamma^2A_j^2=\gamma^2(\kappa_j+\kappa_i)\mu^2$, the same number.  Hence $\kappa_iA_i^2=\kappa_jA_j^2$ at every interface, and repeating this across the checkerboard yields
\begin{equation}\label{eq:amplitude-balance}
 \boxed{
 \kappa_1A_1^2=\kappa_2A_2^2=\kappa_3A_3^2=\kappa_4A_4^2.}
\end{equation}
Equivalently,
\begin{equation}\label{eq:amplitude-ratio}
  A_2^2=A_4^2=R A_1^2=R A_3^2.
\end{equation}
The unweighted solution is therefore strongly asymmetric, but the coefficient-weighted sector amplitude is exactly balanced.  Together with the sector identities, this is already the statement needed for the mesh argument.

\begin{thm}[Hidden energy symmetry]\label{thm:main}
For Kellogg's perpendicular two-phase checkerboard and its homogeneous singular mode, there is a constant $\Lambda>0$ such that
\begin{equation}\label{eq:main-gradient}
  \boxed{\kappa(x)|\nabla u(x)|^2=\Lambda\, r^{2\gamma-2},}
\end{equation}
and
\begin{equation}\label{eq:main-hessian}
  \boxed{\kappa(x)|\gradb^2u(x)|_F^2
  =2(1-\gamma)^2\Lambda\, r^{2\gamma-4}.}
\end{equation}
Both right-hand sides are functions of $r$ alone.  Neither the angular variable nor the coefficient appears there, although each of $\kappa$ and $u$ depends on both.  The identities hold in the four open quadrants, and by the amplitude balance \eqref{eq:amplitude-balance} the one-sided traces of the two scalar quantities agree across every interface, so each admits a unique continuous extension to $\Omega\setminus\{0\}$.
\end{thm}

\begin{proof}
On $Q_i$, \eqref{eq:sector-gradient} and \eqref{eq:sector-hessian} give
\[
\kappa|\nabla u|^2=\kappa_i\gamma^2A_i^2\,r^{2\gamma-2}, \qquad \kappa|\gradb^2u|_F^2=2(1-\gamma)^2\bigl(\kappa_i\gamma^2A_i^2\bigr)r^{2\gamma-4}.
\]
By \eqref{eq:amplitude-balance} the quantity $\kappa_i\gamma^2A_i^2$ takes one common value on the four quadrants; call it $\Lambda$.
\end{proof}

The first identity also shows directly that each quadrant carries the same energy near the origin.  For a disk $B_h$ centered at the origin with $0<h\le1$, so that $B_h\subset\Omega$,
\begin{equation}\label{eq:energy-scale}
 \int_{B_h}\kappa|\nabla u|^2\,\dd x
 =\frac{\pi\Lambda}{\gamma}h^{2\gamma},
\end{equation}
and each quadrant contributes exactly one quarter.

\medskip \noindent\textbf{A byproduct.} The same transport calculation also removes the numerical nonlinear solve from the usual specification of the benchmark.

\begin{cor}[the benchmark constants in closed form]\label{cor:closed}
The system \eqref{eq:cd-system}, with its admissibility restrictions, is solved by
\begin{equation}\label{eq:closed}
  R=\cot^2\Bigl(\frac{\pi\gamma}{4}\Bigr),
  \qquad
  \rho=\frac{\pi}{4},
  \qquad
  \sigma=\frac{\pi}{4}-\frac{\pi}{2\gamma}=-\frac{\arctan\sqrt R}{\gamma} .
\end{equation}
\end{cor}

\begin{proof}
The first identity is \eqref{eq:Rgamma}.  It is equivalent to $\arctan\sqrt R=\frac\pi2-\frac{\pi\gamma}{4}$, which on division by $\gamma$ gives the stated $\sigma$.  It remains to substitute.  With $\rho=\pi/4$ and $\sigma\gamma=-\arctan\sqrt R$ we have $\tan(\sigma\gamma)=-\sqrt R$ and $\cot(\sigma\gamma)=-1/\sqrt R$, so the second equation of \eqref{eq:cd-system} reads $1/R=\tan(\pi\gamma/4)/\sqrt R$, which is \eqref{eq:Rgamma} again.  Next, $(\tfrac\pi2-\sigma)\gamma=\frac\pi2+\frac{\pi\gamma}{4}$, whence $\tan\bigl((\tfrac\pi2-\sigma)\gamma\bigr)=-\cot\frac{\pi\gamma}{4}=-\sqrt R$, and the first equation gives $\sqrt R\cot(\pi\gamma/4)=R$.  Finally $(\tfrac\pi2-\rho)\gamma=\frac{\pi\gamma}{4}$, so the third equation gives $\sqrt R\cot(\pi\gamma/4)=R$ as well.  For admissibility, $0<\gamma<1$ gives $0<\gamma<2$; $2\gamma\rho=\pi\gamma/2$ lies in $(0,\pi\gamma)$; and $-2\gamma\sigma=2\arctan\sqrt R=\pi-\frac{\pi\gamma}{2}$ lies in $\bigl(\pi(1-\gamma),\pi\bigr)$.
\end{proof}

Relation \eqref{eq:Rgamma} is a bijection between $\gamma\in(0,1)$ and $R\in(1,\infty)$.  The coefficient is the data; the exponent is what the data produces,
\begin{equation}\label{eq:gamma-from-R}
  \gamma=\frac4\pi\arctan\frac1{\sqrt R}\ \in(0,1),
\end{equation}
a single number attached to the problem, the same in all four quadrants.

The transfer-matrix argument determines the mode only up to a multiplicative constant, so the common value $\kappa_iA_i^2$ is fixed only once a normalization is chosen.  For the normalization in \eqref{eq:cd-mu} the amplitude on $Q_2$ is $A_2=\cos(\rho\gamma)=\cos(\pi\gamma/4)$, and $\kappa_2=1$, so the common weighted amplitude in \eqref{eq:main-gradient} is
\begin{equation}\label{eq:Lambda-closed}
  \Lambda=\gamma^2\,\frac{R}{1+R}=\gamma^2\cos^2\Bigl(\frac{\pi\gamma}{4}\Bigr),
\end{equation}
a function of $\gamma$ alone.

At $\gamma=0.1$, \eqref{eq:closed} returns $R=161.4476387975881$ and $\sigma=-14.92256510455152$, reproducing \eqref{eq:cd-numbers} to every printed digit.  Finally, \eqref{eq:amplitude-ratio} now reads $|A_2/A_1|=\sqrt R$, which at $\gamma=0.1$ is $12.706\ldots$: the solution is more than an order of magnitude larger in the low-coefficient quadrants, while the weighted amplitude is identical in all four.

\medskip \noindent\textbf{Why the symmetry is exceptional.} The transmission conditions alone do not imply it.  On a separate configuration, let $D=(-1,1)^2$, split by $x_2=0$ into $D_\pm$, fix $q>1$, and set
\[
  \widehat\kappa=\begin{cases}1,&x_2<0,\\ q,&x_2>0,\end{cases}
  \qquad
  w=\begin{cases}x_1^2-x_2^2+x_1+x_2,&x_2<0,\\ x_1^2-x_2^2+x_1+q^{-1}x_2,&x_2>0.\end{cases}
\]
Each branch is harmonic and the two traces on $x_2=0$ are both $x_1^2+x_1$.  Differentiating in the direction $e_2$ on both sides, one has $\partial_2w_-=1$ and $\partial_2w_+=q^{-1}$, so $\widehat\kappa_-\partial_2w_-=1=q\cdot q^{-1}=\widehat\kappa_+\partial_2w_+$.  Thus $[w]=[\widehat\kappa\,\partial_2w]=0$ on $x_2=0$, and $w$, with its own trace as Dirichlet data on $\partial D$, solves the corresponding bounded-domain transmission problem.  Its Hessian is $\operatorname{diag}(2,-2)$ on both sides, so $|\nabla^2w|_F^2=8$ on each of them, and, with the Hessian read piecewise on $D_\pm$,
\begin{equation}\label{eq:counterexample}
 \widehat\kappa\,|\nabla^2w|_F^2=
 \begin{cases}
 8,&x_2<0,\\
 8q,&x_2>0.
 \end{cases}
\end{equation}
The quantity that drives the mesh is therefore genuinely unbalanced across this interface.  For shape-regular elements of local diameters $h_-$ and $h_+$ on the two sides, the same $s_K=1$ local scale gives $\varepsilon_-^2\simeq8h_-^4$ and $\varepsilon_+^2\simeq8q\,h_+^4$, so equal-error allocation requires $h_+\simeq q^{-1/4}h_-$: for this problem the correct mesh is material biased.  Figure~\ref{fig:quadratic-interface-mesh} reports the corresponding computation.  What is exceptional in the checkerboard is the extra global matching that removes the coefficient from the weighted quantities altogether.

\section{From hidden symmetry to the radial adaptive mesh}
\label{sec:mesh}

We now turn the continuous identities into a mesh statement.  The standard adaptive principle of equal discretization-error distribution predicts that an adapted mesh makes the local approximation scales \eqref{eq:local-scale} comparable,
\begin{equation}\label{eq:equidistribution}
  \varepsilon_K\simeq\varepsilon
  \qquad(K\in\Tcal),
\end{equation}
with no exact equality on a finite triangulation intended.  For the Kellogg solution this leads to the spatial refinement law derived below.  We do not claim, and do not need, that the resulting mesh is a global best-$N$ approximation; the statement is the spatial prediction of \eqref{eq:equidistribution}.

\medskip \noindent\textbf{Convention on constants.} From here on $A\simeq B$ means $cB\le A\le CB$, and $A\lesssim B$ means $A\le CB$.  The hidden constants are independent of the mesh, of the element, of its position, and of the quadrant; in particular, once the energy weight $\kappa^{1/2}$ has been included, no additional local coefficient factor such as $R^{\pm1/2}$ remains.  They may depend on the fixed member of the Kellogg family, equivalently on $\gamma$ or on $R$, on the shape-regularity parameter, and, where fractional regularity is used, on the fixed exponent $s\in(0,\gamma)$ and the reference element shape.  No uniformity as $R\to\infty$ is intended, since $R$ and $\gamma$ are linked by \eqref{eq:gamma-from-R} and the solution changes with them.

\subsection{Elements away from the singular point}

Let $K$ be interface-fitted with $0\notin\overline K$, and let $r_K=\operatorname{dist}(K,0)$.  Then $u|_K\in H^2(K)$, so $s_K=1$, and \eqref{eq:p1-local} together with Theorem~\ref{thm:main} gives
\begin{equation}\label{eq:outer-contribution}
 \varepsilon_K^2
 =h_K^2\int_K\kappa|\nabla^2u|_F^2\,\dd x
 =2(1-\gamma)^2\Lambda\,h_K^2\int_K r^{2\gamma-4}\,\dd x
 \simeq h_K^4r_K^{2\gamma-4},
\end{equation}
using $r\simeq r_K$ on $K$ and shape regularity in the form $|K|\simeq h_K^2$.  Hence
\begin{equation}\label{eq:outer-error}
 \boxed{\varepsilon_K\simeq h_K^2r_K^{\gamma-2}.}
\end{equation}
No quadrant-dependent factor remains, and by Theorem~\ref{thm:main} no factor attached to the interfaces remains either.  Elements at the same radial scale must therefore have the same asymptotic size if their local errors are to be equidistributed.

Writing $n(r)\simeq h(r)^{-2}$ for the local element density and solving \eqref{eq:outer-error} for $h$, so that $h(r)^2\simeq\varepsilon\,r^{2-\gamma}$, condition \eqref{eq:equidistribution} gives
\begin{equation}\label{eq:density-from-equi}
  n(r)\simeq\varepsilon^{-1}r^{\gamma-2}.
\end{equation}
Equivalently $h(r)\simeq\varepsilon^{1/2}r^{1-\gamma/2}$.  We keep both in terms of $\varepsilon$ for the moment; the conversion to the element budget is made in section~\ref{sec:budget}, once the scale of the innermost elements is known.

At every fixed radius the predicted size is independent of direction, so in the continuous equidistribution model the density may be written $n_*(r)=c_*\varepsilon^{-1}r^{\gamma-2}$.  Hence for any angular interval $I=(\theta_1,\theta_2)$ and any centered annulus $r_1<r<r_2$ contained in $B_1$,
\begin{equation}\label{eq:annular-balance}
  \Ncal_I(r_1,r_2):=\int_{\theta_1}^{\theta_2}\!\!\int_{r_1}^{r_2}n_*(r)\,r\dd r\dd\theta
  =\frac{c_*\varepsilon^{-1}}{\gamma}\,|I|\,\bigl(r_2^\gamma-r_1^\gamma\bigr).
\end{equation}
Equal angular sectors therefore receive equal element mass, and quadrant balance is only the special case $|I|=\pi/2$.

\subsection{Elements attached to the origin}

The only elements that require fractional regularity are the finitely many whose closure contains the origin.  They need it because the $H^2$ seminorm is not integrable there.  Indeed $|\nabla^2u|_F^2\simeq r^{2\gamma-4}$ gives $\int_K|\nabla^2u|_F^2\dd x\simeq\int_0^{h}r^{2\gamma-3}\dd r$, which diverges for $\gamma<1$.  Let $K_{h_i,i}=h_iK_{1,i}$, $i=1,\ldots,4$, where the four reference triangles $K_{1,i}$ are copies of one another, obtained by rotating through multiples of $\pi/2$, each contained in $\overline{Q_i}$ with the origin as one vertex.  The four sizes $h_i$ are left independent; equidistribution will force them to be comparable.  For every fixed $s\in(0,\gamma)$ one has $u|_{K_{h_i,i}}\in H^{1+s}(K_{h_i,i})$, and the appropriate local scale is $\varepsilon_{K_{h_i,i}}=h_i^s\kappa_i^{1/2}|\nabla u_i|_{s,K_{h_i,i}}$.

The homogeneity of $\nabla u$ gives the fractional-seminorm scaling
\begin{equation}\label{eq:fractional-scale}
 |\nabla u_i|_{s,K_{h_i,i}}
 =h_i^{\gamma-s}|\nabla u_i|_{s,K_{1,i}},
\end{equation}
since for a field homogeneous of degree $\gamma-1$ the \emph{squared} Slobodeckij seminorm picks up $h_i^{2\gamma-2s}$ under $x=h_i\xi$ in two dimensions.

It remains to compare the four constants, and for this the scalar amplitude balance is enough.  Rotating $K_{1,i}$ onto $K_{1,1}$ turns $u_i$ into $A_ir^\gamma\cos(\gamma(\theta-\theta_i'))$ for some phase $\theta_i'$, so the four rotated solutions differ only in the amplitude $A_i$ and in a phase. Writing such a function as $\operatorname{Re}(\lambda z^\gamma)$ with $|\lambda|=|A_i|$, its gradient is $|A_i|$ times a fixed field composed with a \emph{constant} rotation, because a shift of phase multiplies $\lambda$ by a unit complex number.  A constant rotation is an isometry, so it leaves distances, the area measure, and the vector differences $|\mathbf v(x)-\mathbf v(y)|$ unchanged, and therefore leaves the seminorm unchanged:
\begin{equation}\label{eq:fractional-amplitude}
  |\nabla u_i|_{s,K_{1,i}}=|A_i|\,S_{s,K_1},
\end{equation}
with $S_{s,K_1}$ independent of $i$.

Multiplying by $\kappa_i^{1/2}$ and using \eqref{eq:amplitude-balance},
\[
\kappa_i^{1/2}|\nabla u_i|_{s,K_{1,i}} =S_{s,K_1}\sqrt{\kappa_iA_i^2} =:C_{s,K_1},
\]
one constant independent of $i$.  Combining this with \eqref{eq:fractional-scale}, the powers of $h_i$ add to $h_i^s\cdot h_i^{\gamma-s}=h_i^\gamma$, so
\begin{equation}\label{eq:origin-balance}
 \boxed{
 \varepsilon_{K_{h_i,i}}
 =C_{s,K_1}\,h_i^{\gamma},
 \qquad i=1,\ldots,4.}
\end{equation}
Only three properties of the Slobodeckij seminorm were used: it is invariant under rigid motions, it scales linearly when the field is multiplied by a constant, and it obeys the dilation scaling \eqref{eq:fractional-scale}.  The same conclusion therefore holds for any coefficient-weighted local measure of $\sqrt\kappa\,\nabla u$ with those three properties.

The coefficient-weighted constant is the same in all four quadrants.  Imposing \eqref{eq:equidistribution} on these elements gives $C_{s,K_1}h_i^{\gamma}\simeq\varepsilon$ for each $i$, hence
\begin{equation}\label{eq:hmin}
  \boxed{h_i\simeq h_{\min}:=\varepsilon^{1/\gamma},\qquad i=1,\ldots,4 .}
\end{equation}
The origin-attached elements of a general shape-regular triangulation need not form such a quartet; the same scaling applies to each of them, with uniformly comparable constants.  This also matches the outer law: setting $r\simeq h$ in $h(r)\simeq\varepsilon^{1/2}r^{1-\gamma/2}$ gives $h\simeq\varepsilon^{1/2}h^{1-\gamma/2}$, that is $h\simeq\varepsilon^{1/\gamma}$ again.

\subsection{The element budget}
\label{sec:budget}

Both cases are now expressed through the single parameter $\varepsilon$, and the count can be closed.  Shape regularity bounds the number of elements meeting the origin, so those contribute $O(1)$ elements.  Away from them, \eqref{eq:density-from-equi} gives
\[
  N_{\rm out}\simeq\varepsilon^{-1}\int_{h_{\min}}^{1}r^{\gamma-2}\,r\dd r
  =\frac{\varepsilon^{-1}}{\gamma}\bigl(1-h_{\min}^{\gamma}\bigr)
  \simeq\varepsilon^{-1},
\]
the last step because $h_{\min}^{\gamma}=\varepsilon\to0$ by \eqref{eq:hmin}.  On the fixed region $\Omega\setminus B_1$ one has $r\simeq1$ and $n\simeq\varepsilon^{-1}$, so it contributes another $O(\varepsilon^{-1})$ elements and does not change the scaling.  Hence
\begin{equation}\label{eq:common-error}
  N=\#\Tcal\simeq\varepsilon^{-1},
  \qquad\text{that is}\qquad
  \varepsilon\simeq N^{-1},
\end{equation}
and substituting into the outer law and into \eqref{eq:hmin},
\begin{equation}\label{eq:mesh-law}
 \boxed{
 h(r)\simeq N^{-1/2}r^{1-\gamma/2},
 \qquad
 n(r)\simeq N r^{\gamma-2},
 \qquad
 h_{\min}\simeq N^{-1/\gamma}.}
\end{equation}

We have therefore reached the main mesh conclusion using one adaptive principle in both cases.  The exponent $\gamma$ determines how rapidly the mesh must grade toward the interface crossing.  The coefficient-weighted amplitude balance removes the angular variable from that grading altogether.  For Kellogg's checkerboard there is no angular preference at all.

Finally, $N$ equidistributed local contributions of size $\varepsilon\simeq N^{-1}$ give the global coefficient-weighted approximation scale
\[
\left(\sum_{K\in\Tcal}\varepsilon_K^2\right)^{1/2} \simeq N^{-1/2}.
\]
Thus the same calculation recovers the familiar $N^{-1/2}$ scale for $\bigl(\sum_K\varepsilon_K^2\bigr)^{1/2}$, but the rate alone does not contain the spatial information in \eqref{eq:origin-balance} and \eqref{eq:mesh-law}.

\section{Why the mesh picture contains information beyond the rate}
\label{sec:diagnostic}

The hidden symmetry makes the Kellogg mesh a useful visual diagnostic, because it fixes a continuous standard against which a computed mesh can be compared.  The standard is \eqref{eq:mesh-law}: the target density is a function of $r$, and by Theorem~\ref{thm:main} it stays one across the interfaces.

An interface-fitted triangulation of course contains the axes in its skeleton, since they are unions of edges by construction; what the identities forbid is any concentration of elements along them.  A correct Kellogg mesh looks like a mesh built for a radially symmetric problem, even though the problem has a discontinuous coefficient and an asymmetric solution.  The picture is therefore not a separate empirical rule.  It is what the theory predicts, read from the mesh rather than from a convergence plot.

Angular structure in the mesh therefore has a specific interpretation.  By \eqref{eq:main-hessian} and \eqref{eq:outer-error} the coefficient-weighted local approximation difficulty at a fixed positive radius is the same in every direction, so a method that marks the axes, or that places systematically more elements on one side of them, is not responding to a feature of the coefficient-weighted local approximation distribution: that distribution has no angular feature to respond to.  A local quantity carrying the wrong power of the coefficient does have one.  The unweighted quantities available here are asymmetric in either direction: by \eqref{eq:amplitude-ratio} the raw gradient is larger by a factor $\sqrt R$ in the low-coefficient quadrants, while dropping the reciprocal coefficient scaling from a residual edge term multiplies its squared contribution by the corresponding edge diffusion scale and therefore favors the high-coefficient ones.  Which way the bias points depends on which power is wrong; what does not depend on it is that an angular dependence has been introduced where the energy norm has none, and that the resulting bias is persistent, the factor being independent of the mesh size.

Figure~\ref{fig:opening-meshes} shows the second of these two cases.  Panel (a) is what the prediction looks like when it is met: the refinement is concentric, and there is no visible excess refinement along the coefficient interfaces.  Panel (b) fails on both counts, and comparing element counts by quadrant puts a number on the failure, $8.35$ against $1.03$ at comparable budgets.  That count is only a convenient scalar test, and a coarse one: it detects a preference between opposite pairs of quadrants and would report nothing about refinement concentrated along the axes themselves, which the concentric standard also forbids.  Conversely, an absence of angular structure does not by itself certify a method, since the radial grading has to be correct as well.

The other failure mode noted in section~\ref{sec:intro}, refinement following the coefficient interfaces, is of a different kind.  It is not a question of the coefficient weight at all.  The transmission condition is $[\kappa\partial_nu]=0$, so the physical flux is continuous while $\nabla u$ itself has a normal jump.  A recovery into a space of continuous vector fields cannot represent that jump, and registers a feature of the exact solution as discretization error; the remedy is to recover the physical flux instead \cite{CZ09}.  Estimators built on jumps of the discrete normal flux do not have this particular recovery-space defect.  The two failures therefore have different causes: the angular bias comes from an incorrect coefficient weighting in the local measure, and is the subtle one, since the estimator may otherwise be entirely reasonable; refinement along the interfaces comes from a recovery space that does not respect the transmission structure of the exact gradient, and is the cruder one.

The aim is not to classify estimators by their mesh pictures but to fix the continuous spatial standard against which a Kellogg computation can be read, and such a standard is needed because the usual report does not supply one.  The Kellogg example is most often presented through its convergence history, and as an answer to the question of rate optimality that report is complete.  It is silent on the present question: the optimal global slope $N^{-1/2}$ cannot by itself certify the spatial allocation of the degrees of freedom, which is the distinction recorded in \eqref{eq:rate-not-mesh}.  The mesh picture contains precisely that missing information.

The radial standard is also exceptional, and the counterexample \eqref{eq:counterexample} shows why.  There the solution and the physical flux are continuous, yet the coefficient-weighted difficulty jumps by the factor $q$ across the interface, and the correct mesh is correspondingly material biased.  A radial mesh would be the wrong one for that problem.  What the checkerboard supplies is the extra relation \eqref{eq:interface-eigenvector} between $\mu$ and $\kappa\mu'$, forced at all four interfaces at once by the sector transport of section~\ref{sec:symmetry} together with the closing condition \eqref{eq:closing}.

This is what Figure~\ref{fig:quadratic-interface-mesh} reports.  Writing $N_\pm$ for the numbers of elements in the two half-squares, which have equal area, $N_+/N_-$ is also the ratio of the mean element densities, and $h_+\simeq q^{-1/4}h_-$ predicts $N_+/N_-\simeq\sqrt q=4$ against the computed $3.934$.  The estimator that produces the radial Kellogg mesh therefore produces a strongly material-biased one here.  On the checkerboard, radiality is what robustness looks like, so a mesh that is not radial is visible evidence that the computation is responding to something absent from the correct local difficulty; which of the two failures it is can then be read from the shape of the departure.  What does not transfer is the picture.  For another problem the spatial standard has to be worked out from that problem's own coefficient-weighted local difficulty before a computed mesh can be read at all.

For the Kellogg checkerboard, robust error estimation, locally optimal approximation, and mesh allocation all rest on the same coefficient weighting.  The coefficient-weighted local difficulty is radial, so a correct adaptive mesh should display the corresponding radial refinement and nothing else.  The same balance on the origin-attached elements holds for any local measure of $\sqrt\kappa\,\nabla u$ that is rotation invariant, scales linearly when the field is multiplied by a constant, and has the same dilation scaling.  The familiar concentric Kellogg mesh is what the hidden energy symmetry of the continuous problem looks like on a triangulation.

Related coefficient-weighted structure also arises for Stokes interface problems with piecewise constant viscosity \cite{LZ26}; that setting will be considered separately.

\end{document}